\documentclass[a4paper, 12pt,oneside]{article}
\usepackage[utf8]{inputenc}
\usepackage{cite}

\usepackage[T2A]{fontenc}
\usepackage{hyperref}
\usepackage{anysize}
\usepackage{amsmath}
\usepackage{amsthm}
\usepackage{amssymb}
\usepackage{fancyhdr}
\usepackage{graphicx}
\usepackage{enumitem}
\usepackage{xcolor}

\newcommand{\dt}{\,{\rm d}t}

\begin{document}

\newtheoremstyle{mtstyle} 
    {\topsep  }                    
    {\topsep}                    
    {\itshape}                   
    {}                           
    {\bfseries}                   
    {.}                          
    {.7em}                       
    {}  

\theoremstyle{mtstyle}
\newtheorem{theorem}{Theorem}[section]
\newtheorem{defi}[theorem]{Definition}
\newtheorem{lemma}[theorem]{Lemma}
\newtheorem{exam}[theorem]{Example}
\newtheorem{corollary}[theorem]{corollary}
\renewcommand*{\proofname}{Proof}
{\Large \bf {Inverse problem for Dirac operators with a small delay }}

\begin{center}
{\bf  Neboj\v{s}a Djuri\'{c} and Biljana Vojvodi\'{c}}
\end{center}

\noindent
{\bf Abstract:}
This paper addresses inverse spectral problems associated with Dirac-type operators with a constant delay, specifically when this delay is less than one-third of the interval length. Our research focuses on eigenvalue behavior and operator recovery from spectra. We find that two spectra alone are insufficient to fully recover the potentials. Additionally, we consider the Ambarzumian-type inverse problem for Dirac-type operators with a delay. Our results have significant implications for the study of inverse problems related to the differential operators with a constant delay and may inform future research directions in this field.


\section{Introduction}\label{sec1}

\medskip
The main goal of this paper is to study a two spectra inverse problem of recovering a system of differential equations with a delay.  Specifically, we consider the boundary value problems ${\cal B}_{j}(a,p,q), $ $j=1,2,$ for Dirac-type system  of the form
\begin{equation}\label{1}
By'(x)+Q(x)y(x-a)=\lambda y(x), \quad 0<x<\pi,
\end{equation}
\begin{equation*}
y_1(0) = y_j(\pi) = 0,
\end{equation*}
where 
\begin{equation*}
B = \begin{bmatrix} 0 & 1 \\ -1 & 0 \end{bmatrix},\;\; y(x) = \begin{bmatrix} y_1(x) \\ y_2(x) \end{bmatrix},\;\; Q(x) = \begin{bmatrix} p(x) & q(x) \\ q(x) & -p(x) \end{bmatrix},
\end{equation*}
as $p(x), q(x) \in L_2(0,\pi)$ are complex-valued functions with $Q(x)=0$ on $(0,a)$ and the delay $a \in (0,\frac{\pi}{3})$. Let $\{\lambda_{n,j}\}_{n\in\mathbb{Z}},$ $j=1,2,$ be eigenvalues of ${\cal B}_{j}(a,p,q)$. The inverse problem is formulated as follows. 

\vspace{0.3cm}
\noindent
\textbf{Inverse problem 1.} Given the two spectra $\{\lambda_{n,j}\}_{n\in\mathbb{Z}}$, $j=1,2,$ construct $p(x),$ $q(x).$ 

\vspace{0.3cm}
\noindent
The idea of using two sequences of eigenvalues for a Dirac-type system on a finite interval originated with Gasymov's and Dzabiev's classic paper \cite{Gasymov} where they developed a constructive solution based on transformation operators and obtained the characterization of the spectral data. Two-spectra inverse problems for Dirac-type systems have been considered in \cite{Albeverio, Eckhardt}. Research on inverse problems for the Dirac  operator with a delay started a few years ago and the first result in this direction is the paper \cite{ButerinD}. In that paper the authors restrict themselves to the case of ``large delay'' when the dependence of the characteristic functions on the potential is linear. For the considered case, however, they achieved their objectives by answering a full range of questions regarding uniqueness, solvability, and uniform stability. Very recently, progress has been made in the nonlinear case in the paper \cite{DjuricVo}, by proving that two spectra uniquely determine potentials if $a \in [2\pi/5, \pi/2),$ yet it is not possible in the case when $a \in [\pi/3, 2\pi/5)$. In \cite{WangY}, further research has been made for the case when $a \in [2\pi/5, \pi/2)$ by giving a necessary and sufficient condition for the solvability and formulating the stability of the inverse problem. The question of whether two spectra are enough to uniquely recover potentials for $a < \pi/3$ remained unanswered. 
In this paper, we filled this gap by giving a negative answer to the question of whether the unique construction of the potentials is possible from the two spectra.

The answer to this question will also be valuable for addressing similar issues with other types of differential operators with delays. A major avenue for further research involves boundary value problems associated with differential equations that include two delays. In the papers \cite{VojvodicDV}, \cite{VojvodicVD} authors study four boundary value problems $\mathcal{D}_{j}(P,Q,m), m=0,1,\; j=1,2,$ for Dirac-type system of the form
\begin{equation}\label{2}
    By'(x)+(-1)^{m}P(x)y(x-a_{1})+Q(x)y(x-a_{2})=\lambda y(x), \; 0<x<\pi,
 \end{equation}
\begin{equation*}\
y_{1}(0)=y_{j}(\pi)= 0,
\end{equation*}
where $\frac{\pi}{3}\leq a_1<a_2<\pi,$
\begin{equation*}
P(x)=\begin{bmatrix}
p_{1}(x) & p_{2}(x)\\
p_{2}(x) & -p_{1}(x)
\end{bmatrix},\hspace{2mm}
Q(x)=\begin{bmatrix}
q_{1}(x) & q_{2}(x)\\
q_{2}(x) & -q_{1}(x)
\end{bmatrix},
\end{equation*}
and $p_{1}(x),p_{2}(x),q_{1}(x), q_{2}(x)\in L^{2}[0,\pi]$ are complex-valued functions such that
\[ P(x)=0, \hspace{2mm} x\in(0,a_{1}), \hspace{2mm} Q(x)=0, \hspace{2mm}  x\in (0,a_{2}).\]

\noindent
Let $\{\lambda_{n,j}^{m}\}$ be the spectrum of the boundary value problem $\mathcal{D}_{j}(P,Q,m),$ $j=1,2,$ $m=0,1,$ and let us assume that delays $a_{1}$ and $a_{2}$ are known. In the papers \cite{VojvodicDV} and \cite{VojvodicVD} the authors consider the following inverse problem.

\vspace{0.3cm}
\noindent
{\bf Inverse Problem 2.} Given the four spectra $\{\lambda_{n,j}^{m}\}_{n\in\mathbb{Z}},$ $j=1,2,$ $m=0,1$, construct the matrix-functions $P(x)$ and $Q(x).$

\vspace{0.3cm}
\noindent
It is obvious that equation (\ref{2}) becomes equation (\ref{1}) when $P(x)=0$ and then Inverse problem 2 reduces to Inverse problem 1. It has been proved in the  paper \cite{VojvodicVD} that Inverse problem 2 has a unique solution as soon as the delays $a_1, a_2 \in [2\pi/5,\pi).$  One can easily conclude that the solution of Inverse problem 2 is not unique in the case $a_2 \in [\pi/3, 2\pi/5).$  Indeed, it has been proved in the paper \cite{DjuricVo} that the solution of Inverse problem 1 is not unique for $a \in [\pi/3,2\pi/5)$, therefore, taking $P(x)=0$ in the equation (\ref{2}), we obtain the same conclusion for Inverse problem 2. The case $a_1 \in [\pi/3, 2\pi/5)$ and $a_2 \in [2\pi/5, \pi)$ has been considered in the paper\cite{VojvodicDV}. It  has been proved that Inverse problem 2 has a unique solution if $2a_1+a_2/2 \geq \pi$, while in the case $2a_1+a_2/2 < \pi$ it does not. \\
In the paper \cite{WangYBD} the authors introduced Dirac-type operators with a global constant delay on a star graph consisting of $m$ equal edges. They proved that the uniqueness theorem is valid in the case when the delay is greater than one-half of the length of the interval.

A large number of differential equations with delays, predominantly sourced from the literature in biological sciences, await exploration. These encompass different typies of models extending across population biology, physiology, epidemiology, economics, neural networks, and the intricate control mechanisms of mechanical systems. The application of differential equations with delay is covered in books by Erneux \cite{Erneux}, Polyanin \cite{Polyanin} and Smith \cite{Smith}. Also,  Norkin \cite{Norkin} dealt with models in which occurs the Sturm-Liouville equation with a delay 
\begin{equation*}
-y''(x)+q(x)y(x-a)=\lambda y(x), \ 0<x<\pi.
\end{equation*}
For this reason, inverse spectral problems for Sturm-Liouville operators with delay are studied in detail. Some important results from this field can be found in the papers \cite{BondarenkoY, ButerinPY, DjuricB1, DjuricB2, DjuricB3, FreilingY, Pikula, PikulaVV, VladicicBV}. In addition, inverse problems for Sturm-Liouville operators with two delays have been thoroughly investigated in papers \cite{VojvodicKC, VojvodicV, VojvodicV2}.

The paper is organized as follows. In Section 2 we describe characteristic functions and study asymptotic behavior of eigenvalues. In Section 3 we construct the class of iso-bispectral potentials  $p(x),q(x)$ that have the same characteristic functions. In this way, we will show that two spectra of boundary value problems ${\cal B}_{j}(a,p,q), $ $j=1,2,$ are not enough to uniquely determine the potentials $p(x),q(x)$. In section 4 we study Ambarzumian-type inverse problem for Dirac operators with constant delay.

\section{Characteristic functions and spectra}\label{sec2}

Let $a \in [\pi/(N+1),\pi/N)$ for $N \in \mathbb{N}$ and $S(x,\lambda)=\begin{bmatrix}
 s_1(x,\lambda) \\ 
 s_2(x,\lambda) \end{bmatrix}$ be the fundamental (vector) solution of Eq. \eqref{1} satisfying the initial condition at the origin:
\[
 S(0,\lambda) = \begin{bmatrix}
 0 \\ 
 -1 \end{bmatrix}.
 \]
The solution $S(x,\lambda)$ is the unique solution of the following integral equation 
\begin{equation}\label{5}
S(x,\lambda) = S_0(x,\lambda) + \int_{a}^{x} Q(t)B^{-1}(\lambda (x-t)) S(t-a,\lambda) \dt 
\end{equation}
where
\[
S_0(x,\lambda)=\begin{bmatrix} 
\sin \lambda x \\ 
-\cos \lambda x \end{bmatrix},
\quad
 B(t)=\begin{bmatrix}
\sin t & \cos t\\
-\cos t & \sin t
\end{bmatrix}.
\]
The method of successive approximations gives
\begin{equation} \label{6}
\begin{split}
S(x,\lambda)&= \sum_{k=0}^{N} S_k(x,\lambda),\\
S_k(x,\lambda)&=\int_{ka}^{x} Q(t)B^{-1}(\lambda (x-t)) S_{k-1}(t-a,\lambda) \dt, \quad x> ka, \\ 
S_k(x,\lambda)&=0, \quad x\leq ka, \quad k=\overline{1,N}.
\end{split}
\end{equation}
According to \eqref{6}, we can calculate (see \cite{DjuricVo})
\begin{equation*}
S_1(x,\lambda)=\int_{a}^{x} Q(t_1) \begin{bmatrix} \cos \lambda (x-2t_1+a) \\ -\sin \lambda (x-2t_1+a) \end{bmatrix} \dt_1,
\end{equation*}
\begin{equation*}
S_2(x,\lambda)=\int_{2a}^{x} Q(t_1)   \int_{a}^{t_1-a} Q(t_2) \begin{bmatrix} \sin \lambda (x-2t_1+2t_2) \\ -\cos \lambda (x-2t_1+2t_2) \end{bmatrix} \dt_2 \dt_1.
\end{equation*}
In a similar way, we obtain
\begin{equation*}
\begin{split}
S_3(x,\lambda)=&\int_{3a}^{x} Q(t_1)\int_{2a}^{t_1-a} Q(t_2) \int_{a}^{t_2-a} Q(t_3) \\
&\begin{bmatrix} \cos \lambda (x-2t_1+2t_2-2t_3+a) \\ -\sin \lambda (x-2t_1+2t_2-2t_3+a) \end{bmatrix} \dt_3 \dt_2 \dt_1,
\end{split}
\end{equation*}
and consequently, by induction, we determine the form of the functions $S_k(x,\lambda)$ for the case $k=n$, $n\in\mathbb{N}$ is even, as 
\begin{equation} \label{7}
\begin{split}
S_n(x,\lambda)=& \int_{na}^{x} Q(t_1)\dt_1 \int_{(n-1)a}^{t_1-a}Q(t_2)\dt_2 \cdots \int_{a}^{t_{n-1}-a} Q(t_{n}) \\ 
&  \begin{bmatrix} \sin \lambda (x-2t_1+2t_2-2t_3+...-2t_{n-1}+2t_{n}) \\ -\cos \lambda (x-2t_1+2t_2-2t_3+...-2t_{n-1}+2t_{n}) \end{bmatrix}
\dt_{n},
\end{split}
\end{equation} 
and for the case $k=m$, $m\in \mathbb{N}$ is odd, as
\begin{equation} \label{8}
\begin{split}
S_m(x,\lambda)=& \int_{ma}^{x} Q(t_1)\dt_1 \int_{(m-1)a}^{t_1-a}Q(t_2)\dt_2 \cdots \int_{a}^{t_{m-1}-a} Q(t_{m})\\ 
&\begin{bmatrix} \cos \lambda (x-2t_1+2t_2-2t_3+...+2t_{m-1}-2t_{m}+a) \\ -\sin \lambda (x-2t_1+2t_2-2t_3+...+2t_{m-1}-2t_{m}+a) \end{bmatrix}
\dt_{m}.
\end{split}
\end{equation}
\\
The functions
\begin{equation}\label{9}
\Delta_j(\lambda):=s_{j}(\pi,\lambda), \ j=1,2, 
\end{equation}
are called the {\it characteristic functions} of the boundary value problems ${\cal B}_{j}(a,p,q)$. These functions are entire in $\lambda$ and their zeros coincide with the eigenvalues of ${\cal B}_{j}(a,p,q)$.

In order to express the characteristic functions in a more convenient way, let us introduce the notation for the functions $S_k (x,\lambda)$ given in Eq. (\ref{6}):
\[S_k(x,\lambda)= \begin{bmatrix}
 s_{1,k}(x,\lambda) \\ 
 s_{2,k}(x,\lambda) \end{bmatrix}, \ k=\overline{1,N}.
\]
Taking  Eq.s \eqref{5}, \eqref{7}, \eqref{8}, and \eqref{9} into account, we can rewrite the characteristic functions as 
\begin{equation}\label{10}
\Delta_j(\lambda)=s_j(\pi,\lambda)=
s_{j,0}(\pi,\lambda) + \sum_{k=1}^{N} s_{j,k}(\pi,\lambda), \ j=1,2. 
\end{equation}
Now let $k=\overline{1,N}$. Then, for the case $k=m$, $m$ is odd, the characteristic functions become
\begin{equation}\label{11}
\begin{split}
s_{j,m}(\pi,\lambda)=& \int_{ma}^{\pi} \dt_1 \int_{(m-1)a}^{t_1-a}\dt_2 \cdots \int_{a}^{t_{m-1}-a} \bigg(Q_{j,1}^{m}(t_1,t_2,\cdots,t_m)\\ 
&\cos \lambda\bigg(\pi-2t_1+2t_2 -2t_3+\cdots+2t_{m-1}- 2t_m+a  \bigg) \bigg)  \dt_m
\\ 
&- \int_{ma}^{\pi} \dt_1 \int_{(m-1)a}^{t_1-a}\dt_2 \cdots \int_{a}^{t_{m-1}-a} Q_{j,2}^{m}(t_1,t_2,\cdots,t_m)\\ 
& \sin\lambda\bigg(\pi-2t_1+2t_2 -2t_3+\cdots+2t_{m-1}- 2t_m+ a \bigg)\bigg)  \dt_m, 
\end{split}
\end{equation}
and for the case $k=n$, $n$ is even, 
\begin{equation} \label{12}
\begin{split}
s_{j,n}(\pi,\lambda)=& \int_{na}^{\pi} \dt_1 \int_{(n-1)a}^{t_1-a}\dt_2 \cdots \int_{a}^{t_{n-1}-a} \bigg(Q_{j,1}^{n}(t_1,t_2,\cdots,t_n) \\ 
&
\sin \lambda\bigg(\pi-2t_1+2t_2 -2t_3+\cdots-2t_{n-1}+ 2t_n  \bigg) \bigg)  \dt_n
\\ 
-& \int_{na}^{\pi} \dt_1 \int_{(n-1)a}^{t_1-a}\dt_2 \cdots \int_{a}^{t_{n-1}-a} Q_{j,2}^{n}(t_1,t_2,...,t_n)\\
& \cos\lambda\bigg(\pi-2t_1+2t_2 -2t_3+...-2t_{n-1}+2t_n \bigg)\bigg)  \dt_n.
\end{split}
\end{equation}
where we use the notation $Q_{j,l}^{k}(t_1,t_2,\cdots,t_k)$ for the entries in the $j-$th row and $l-$th column of the matrix
\[
Q^k(t_1,t_2,\cdots,t_k):= Q(t_1) Q(t_2)\cdots Q(t_k).
\]
One can easily show that next relations hold:
\begin{equation}\label{13}
\begin{array}{ll}
Q_{11}^k(t_1,t_2,\cdots,t_k)=(-1)^k Q_{22}^k(t_1,t_2,\cdots,t_k), \\
Q_{21}^k(t_1,t_2,\cdots,t_k)=(-1)^{k+1} Q_{12}^k(t_1,t_2,\cdots,t_k).
\end{array}
\end{equation}
It is obvious from \eqref{11} and \eqref{12} that the next asymptotic holds
\begin{equation}\label{14}
s_{j,k}(\pi,\lambda)= O \bigg(\exp\big(| \text{Im} \lambda|(\pi-ka)\big)\bigg), \ |\lambda|\to \infty.
\end{equation}
Taking equations \eqref{10},\eqref{11},\eqref{12} and \eqref{14} into account, we derive the asymptotics of the characteristic functions given in Eq. \eqref{10}, for $\ |\lambda|\to \infty,$
\begin{align} \nonumber
\Delta_1(\lambda)=& \sin\lambda\pi + \int_{a}^{\pi} p(t)\cos\lambda (\pi-2t_1+a) \dt_1 \\ \label{15} 
&- \int_{a}^{\pi} q(t)\sin\lambda (\pi-2t_1+a) \dt_1 +  O \bigg(\exp\big(| \text{Im} \lambda|(\pi-2a)\big)\bigg), \\ \nonumber
\Delta_2(\lambda)=& -\cos\lambda\pi + \int_{a}^{\pi} q(t)\cos\lambda (\pi-2t_1+a) \dt_1 \\ \label{16}
&+ \int_{a}^{\pi} p(t)\sin\lambda (\pi-2t_1+a) \dt_1 + O \bigg(\exp\big(| \text{Im}\lambda|(\pi-2a)\big)\bigg).
\end{align}
Therefore, the following theorem can be proved by using the standard method \cite[Lemma 3]{ButerinPY}, Rouche's theorem, and Eq.s \eqref{15}, \eqref{16}.
\begin{theorem}\label{th1}
The boundary value problems ${\cal B}_{j}(a,p,q),$ $j=1,2,$ have infinitely many eigenvalues
$\{\lambda_{n,j}\}_{n\in\mathbb{Z}}$ of the form
\begin{equation*}\label{17}
\lambda_{n,j}=n+\frac{1-j}{2}+ o(1),\ |n|\to \infty.
\end{equation*}
\end{theorem}
\begin{lemma}\label{lem1}
The specification of the spectrum $\{\lambda_{n,j}\}_{n\in{\mathbb Z}}$,  $j=1,2,$
uniquely determines the characteristic function $\Delta_j(\lambda)$ by the formulae
\begin{equation*}\label{18}
\Delta_1(\lambda)=\pi(\lambda_{0,1}-\lambda)\prod_{|n|\in{\mathbb N}}\frac{\lambda_{n,1}-\lambda}n \exp\Big(\frac\lambda{n}\Big), \quad
\Delta_2(\lambda)=\prod_{n\in{\mathbb Z}}\frac{\lambda_{n,2}-\lambda}{n-1/2} \exp\Big(\frac\lambda{n-1/2}\Big).
\end{equation*}
\end{lemma}
\begin{proof}
The proof can be done similarly as in \cite[Theorem 5]{Buterin} by using the Hadamard's factorization method. 
\end{proof}

 In addition to examining the asymptotics of the characteristic functions, solving the Ambarzumian-type inverse problem requires considering appropriate combinations of these functions and determining their asymptotic behavior. For that purpose let us introduce functions:

\begin{equation}\label{19}
\mathcal{L}(\lambda)=\dfrac{1}{2}\bigg(\Delta_1(\lambda)+\Delta_1(-\lambda) + i\big(\Delta_2(\lambda)-\Delta_2(-\lambda)\big)\bigg),
\end{equation}
\begin{equation}\label{20}
\mathcal{M}(\lambda)=\exp (i \lambda \pi)+\dfrac{1}{2} \bigg(\Delta_2(\lambda)+\Delta_2(-\lambda)+i\big(-\Delta_1(\lambda)+\Delta_1(-\lambda)\big)\bigg).
\end{equation} 
It follows from Eq.s \eqref{10}, \eqref{11}, \eqref{12}, \eqref{13}, \eqref{19} and \eqref{20} that
\begin{equation*} \label{21}
\mathcal{L} (\lambda) = \sum_{k=1}^{N} \mathcal{L}_k (\lambda), \;\;\;\ \mathcal{M} (\lambda)= \sum_{k=1}^{N} \mathcal{M}_k (\lambda), 
\end{equation*}
holds for
\begin{equation} \label{22}
\begin{split}
\mathcal{L}_k(\lambda)&=\int_{ka}^{\pi} \dt_1 \int_{(k-1)a}^{t_1-a}\dt_2 \cdots \int_{a}^{t_{k-1}-a} \bigg(Q_{2-\varphi(k),1}^{k}(t_1,t_2,\cdots,t_k) \\  
&\exp i\lambda\big(\pi-2t_1+2t_2 -2t_3+\cdots+(-1)^{k-1}2t_{k-1}+ (-1)^k 2t_k +a \varphi(k) \big) \bigg)  \dt_k,
\end{split}
\end{equation}
\begin{equation}\label{23}
\begin{split}
\mathcal{M}_k(\lambda)&=(-1)^{k+1}\int_{ka}^{\pi} \dt_1 \int_{(k-1)a}^{t_1-a}\dt_2 \cdots\int_{a}^{t_{k-1}-a} \bigg(Q_{1+\varphi(k),1}^{k}(t_1,t_2,...,t_k) \\  
&\exp i\lambda\big(\pi-2t_1+2t_2 -2t_3+...+(-1)^{k-1}2t_{k-1}+ (-1)^k 2t_k +a \varphi(k) \big) \bigg)  \dt_k,
\end{split}
\end{equation}
where $\varphi(2k)=0$ and $\varphi(2k-1)=1, \;k\in  \mathbb{N}.$ In particular, the following asymptotic formulas are valid for $\text{Im} \lambda>0$ as $|\lambda| \to \infty$:
\begin{equation*}
\mathcal{L}_k(\lambda)=O\big(\exp(-i\lambda(\pi-ka))\big), \ \mathcal{M}_k(\lambda)=O\big(\exp(-i\lambda(\pi-ka))\big).
\end{equation*}

\section{Iso-bispectral potentials}\label{sec3}

In this section we shall establish non-uniqueness of the solution of Inverse problem 1. Specifically, we construct an infinite family of iso-bispectral potentials $p_{\alpha}(x)$ and $q_{\beta}(x)$, $\alpha,\beta\in \mathbb{C}$, i.e. for which both problems ${\cal B}_{1}(a,p_{\alpha},q_{\beta})$  and ${\cal B}_{2}(a,p_{\alpha},q_{\beta})$ posses one and the same pair of spectra. For that purpose we use the ideas of the construction of the counterexample from the paper \cite{DjuricB3} and \cite{DjuricVo}. We restrict ourselves with potentials vanishing on the interval $(3a,\pi)$ and using (\ref{10},\ref{11},\ref{12}) we obtain the following representation of the characteristic functions
\begin{equation*}
\begin{split}
\Delta_1(\lambda):=&\sin \lambda\pi +\int_{a}^{3a} p(t)\cos \lambda(\pi-2t+a)\,dt -  \int_{a}^{3a} q(t)\sin \lambda(\pi-2t+a)\,dt \\
&+ \int_{2a}^{3a}  \int_{a}^{t-a} \big(p(t)p(s) + q(t)q(s) \big) \sin \lambda (\pi-2t+2s)\;ds\; dt\\
&+ \int_{2a}^{3a}  \int_{a}^{t-a} \big(q(t)p(s)-p(t)q(s)\big) \cos \lambda (\pi-2t+2s)\;ds\; dt,
\end{split}
\end{equation*}

\begin{equation*}
\begin{split}
\Delta_2(\lambda):=&-\cos \lambda\pi + \int_{a}^{3a} p(t) \sin \lambda(\pi-2t+a)\,dt  + \int_{a}^{3a} q(t)\cos \lambda(\pi-2t+a)\;dt \\
&- \int_{2a}^{3a}  \int_{a}^{t-a} \big( p(t)p(s)+q(t)q(s) \big) \cos \lambda (\pi-2t+2s)\;ds\; dt ,\\
&+ \int_{2a}^{3a}  \int_{a}^{t-a} \big(q(t)p(s)- p(t)q(s)\big) \sin \lambda (\pi-2t+2s)\;ds\; dt.
\end{split}
\end{equation*}
Let us assume that the delay $a$ is known. For fixed $a\in (0,\frac{\pi}{3})$ let us define the integral operator 
\begin{equation*}
\begin{split}
M_hf(x)=\int_{3a/2}^{7a/2-x} f(t) h\left(t+x-\frac{a}{2}\right) dt, \;\;x\in \left(\frac{3a}{2},2a\right),
\end{split}
\end{equation*}
for some non-zero real function $h(x) \in L_2(\frac{5a}{2},3a).$ 
\\
Operator $M_h$ is a non-zero compact Hermitian operator in $L_2(\frac{3a}{2},2a)$ and hence, it has at least one non-zero eigenvalue $\eta.$ Putting $h_m(x)=(-1)^mh(x)/\eta$ for $m\in \{0,1\}$, we obtain that $(-1)^m$ is an eigenvalue of the operator $M_{h_m}.$ Let $e_m(x)$ be a corresponding eigenfunction, i.e.
\begin{equation*}\label{31}
M_{h_m}(x)=(-1)^me_m(x), \;\; x\in (\frac{3a}{2},2a).
\end{equation*}
Now we construct the iso-bispectral family of functions
\begin{equation*}
\begin{split}
D=\{p_\alpha(x), q_\beta(x):\alpha,\beta\in \mathbb{C}\}
\end{split}
\end{equation*}
where
\begin{equation}\label{32}
p_\alpha(x)=\begin{cases} 0, &  x\in (0,\frac{3a}{2})\cup (2a,\pi),\\ 
\alpha e_1(x), &  x\in (\frac{3a}{2},2a), \end{cases}   
\end{equation} 
and
\begin{equation}\label{33}
q_\beta(x)=\begin{cases} 0,&  x\in (0,\frac{3a}{2})\cup (2a,\frac{5a}{2})\cup(3a,\pi),\\ 
\beta e_0(x), &  x\in (\frac{3a}{2},2a), \\
h(x), & x\in (\frac{5a}{2},3a).   \end{cases}  
\end{equation} 
We shall prove that the spectra of boundary value problem ${\cal B}_j(a,p_\alpha, q_\beta)$ for $j=1,2,$ is independent of $\alpha$ and $\beta$, hence, the solution of the Inverse problem 1 is not unique for the delay less than one-third of the interval length. 

\begin{theorem}
Let delay $a\in (0,\frac{\pi}{3}).$  The spectra 
$\{\lambda_{n,j}\}_{n\in\mathbb{Z}}$ of boundary value problem ${\cal B}_j(a,p_\alpha, q_\beta)$ for $j=1,2,$ is independent of $\alpha$ and $\beta$. 
\end{theorem}
\begin{proof}
Taking into account that all multiple integrals except the double one in \eqref{11} and \eqref{12} are equal to zero, we obtain that the characteristic functions for the boundary value problem ${\cal B}_j(a,p_\alpha, q_\beta)$ have the form
\begin{equation}\label{34}
\Delta_1(\lambda)=\sin \lambda \pi + \int_{a/2}^{5a/2} K_1(x) \cos \lambda (\pi-2x)dx -  \int_{a/2}^{5a/2} K_2(x) \sin \lambda (\pi-2x)dx,
\end{equation}
\begin{equation}\label{35}
\Delta_2(\lambda)=-\cos \lambda \pi + \int_{a/2}^{5a/2} K_1(x) \sin \lambda (\pi-2x)dx +  \int_{a/2}^{5a/2} K_2(x) \cos \lambda (\pi-2x)dx,
\end{equation}
where 
\begin{equation*}
K_1(x)=\begin{cases} p(x+a/2), \hspace{1cm} x\in (a/2,a)\cup (2a,5a/2),\\ 
p(x+a/2)-\int\limits_{x+a}^{3a} p(t)q(t-x)dt +\int\limits_{x+a}^{3a} q(t)p(t-x)dt , \;\; x\in (a,2a), \end{cases} 
\end{equation*}

\begin{equation*}
K_2(x)=\begin{cases} q(x+a/2), \hspace{1cm} x\in (a/2,a)\cup (2a,5a/2),\\ 
q(x+a/2)-\int\limits_{x+a}^{3a} p(t)p(t-x)dt - \int\limits_{x+a}^{3a} q(t)q(t-x)dt , \;\; x\in (a,2a). \end{cases}
\end{equation*}
Then, from (\ref{32}) and (\ref{33}) we obtain 
\begin{equation*}
K_1(x-a/2)=\begin{cases} 0, &  x\in (0,\frac{3a}{2})\cup (2a,3a),\\ 
p_\alpha(x)+\int\limits_{3a/2}^{7a/2-x} p_\alpha(t)h(t+x-a/2)dt , &  x\in (\frac{3a}{2},2a), \end{cases} 
\end{equation*}
\begin{equation*}
K_2(x-a/2)=\begin{cases} 0, &  x\in (0,\frac{3a}{2})\cup (2a,\frac{5a}{2}),\\ 
q_\beta(x)-\int\limits_{3a/2}^{7a/2-x} q_\beta(t)h(t+x-a/2)dt , &  x\in (\frac{3a}{2},2a), \\
h(x), & x\in (\frac{5a}{2},3a).   \end{cases}
\end{equation*}
Since $e_0(x)$ and $e_1(x)$ are eigenfunctions of the operator $M_h$ corresponding to the eigenvalues $1$ and $-1$ respectively, we obtain
\begin{equation*}
K_1(x-a/2)=0, \;\;  x\in (a,3a),
\end{equation*}
and 
\begin{equation*}
K_2(x-a/2)=\begin{cases} 0,\ &  x\in (a,\frac{5a}{2}), \\
h(x), &  x\in (\frac{5a}{2},3a).   \end{cases}
\end{equation*}
Then, using (\ref{34}) and (\ref{35}) we conclude that characteristic functions for family of iso-bispectral potentials $D$ have the form
\begin{equation*}
\Delta_1(\lambda)=\sin \lambda \pi -  \int_{5a/2}^{3a} h(x) \sin \lambda (\pi-2x+a)dx,
\end{equation*}

\begin{equation*}
\Delta_2(\lambda)=-\cos \lambda \pi +  \int_{5a/2}^{3a} h(x) \cos \lambda (\pi-2x+a)dx,
\end{equation*}
i.e. they are independent of $\alpha$ and $\beta$.
\end{proof}

\section{Ambarzumian-type inverse problem}\label{sec4}

In Section 3, we proved that the solution to Inverse Problem 1 is not unique. Using the results from Section 3 and the paper \cite{DjuricVo}, we conclude that two spectra are not sufficient to recover the potentials when \( a \in (0, 2\pi/5) \). However, it is interesting to examine whether the uniqueness of the solution holds for the Ambarzumian-type inverse problem for Dirac operators with delay. In the papers \cite{ButerinPY, FreilingY}, the authors prove that the uniqueness of the solution holds for the Ambarzumian-type inverse problem for Sturm-Liouville operators with delay. For this reason, we consider boundary value problems \(\widetilde{\mathcal{B}}_j = \mathcal{B}_j(\widetilde{p}, \widetilde{q}, a)\), where \(\widetilde{p}(x) = \widetilde{q}(x) = 0\) for \( j = 1, 2 \). Since it is known that for \( a \in [2\pi/5, \pi) \), two spectra uniquely determine the potentials \( p(x) \) and \( q(x) \) (see \cite{ButerinD, DjuricVo}), we will further consider the case \( a \in (0, 2\pi/5) \).
\\
Let $\{\widetilde{\lambda}_{n,j}\}_{n \in \mathbb{Z}},$  be the eigenvalues of boundary value problems 
$\widetilde{\mathcal{B}}_j=\mathcal{B}_j(\widetilde{p},\widetilde{q},a),j=1,2.$
Denote by
$\widetilde{\Delta}_j(\lambda)$ the characteristic function of 
$\widetilde{\mathcal{B}}_j,\;j=1,2$. From (\ref{10}),(\ref{11}) and (\ref{12}) we obtain that 
$\widetilde{\Delta}_1(\lambda)=\sin \lambda \pi$ and $\widetilde{\Delta}_2(\lambda)=-\cos \lambda \pi.$ Therefore, it is obvious that $\widetilde{\lambda}_{n,j}=n-j/2,$ $n \in \mathbb{Z}, \;j=1,2.$
\medskip
\noindent
\begin{theorem}\label{thm41}
Let the delay $a\in (0,\frac{2\pi}{5})$ is known. If $\lambda_{n,j}=\widetilde{\lambda}_{n,j}$ for all $n \in \mathbb{Z}$, $j=1,2,$ then $p(x)=q(x)=0$ a.e. on $(a,\pi).$ 
\end{theorem}
\begin{proof}
Let us assume that $\lambda_{n,j}$=$\widetilde{\lambda}_{n,j},$ $n\in\mathbb{Z},$ $j=1,2,$ holds. By virtue of Lemma \ref{lem1} one has
\[\Delta_1(\lambda)=\sin \lambda\pi, \quad  \Delta_2(\lambda)=-\cos \lambda\pi,\]
and consequently $\mathcal{L}(\lambda)=0$ and $\mathcal{M}(\lambda)=0.$ From (\ref{19}) it follows that
\begin{equation}\label{25}
\mathcal{L}_1(\lambda) = -\mathcal{L}^+(\lambda), \quad\mathcal{M}_1(\lambda) = -\mathcal{M}^+(\lambda),
\end{equation}
where for $k\geq2,$
\[\mathcal{L}^+(\lambda)=\sum_{k=2}^{N} \mathcal{L}_k(\lambda), \quad \mathcal{M}^+(\lambda)=\sum_{k=2}^{N} \mathcal{M}_k(\lambda),\] 
and for $k=1,$  $\mathcal{L}^+(\lambda)=0$, $\mathcal{M}^+(\lambda)=0$.
Let notice that
\begin{equation}\label{26}
\begin{split}
\mathcal{L}_1(\lambda)&=\int_{a}^{\pi} p(t_1) \exp \big(i \lambda (\pi-2t_1+a)\big) dt_1, \\ \mathcal{M}_1(\lambda)&=\int_{a}^{\pi} q(t_1) \exp \big(i \lambda (\pi-2t_1+a)\big) dt_1.
\end{split}
\end{equation}

\noindent
Now we divide the proof of the theorem in three steps. 

\noindent
\paragraph{Step (1):} Consider the case when \( p(x) = q(x) = 0 \) a.e. on the interval \( (2a, \pi) \). In this scenario, we have \(\mathcal{L}^+ (\lambda) = 0\) and \(\mathcal{M}^+ (\lambda) = 0\). From equation \((\ref{25})\), it follows that \(\mathcal{L}_1 (\lambda) = \mathcal{M}_1 (\lambda) = 0\). Consequently, from equation \((\ref{26})\), we deduce that \( p(x) = q(x) = 0 \) a.e. on the interval \( (a, \pi) \).

\noindent
\paragraph{Step (2):} In this step, we observe the case when \( p(x) = q(x) = 0 \) a.e. on the interval \( (\pi - \nu a / 2, \pi) \) for a fixed \(\nu = \overline{0, 2N-3}\). We will show that in this case \( p(x) = q(x) = 0 \) a.e. on the interval \( (\pi - (\nu + 1)a / 2, \pi) \). We consider the case \(\pi - \nu a / 2 > 2a\); otherwise, we arrive at Step (1) and the proof is complete. Then from  (\ref{22},\ref{23}) we obtain
\[
\begin{split}
\mathcal{L}_2(\lambda)=\int_{2a}^{\pi-\nu a/2} dt_1 \int_{a}^{t_1-a} Q_{2,1}^{2}(t_1,t_2) \exp \big(-i \lambda (-\pi+2t_1-2t_2)\big) dt_2, \\
\mathcal{M}_2(\lambda)=\int_{2a}^{\pi-\nu a/2} dt_1 \int_{a}^{t_1-a} Q_{1,1}^{2}(t_1,t_2) \exp \big(-i \lambda (-\pi+2t_1-2t_2)\big) dt_2
\end{split}
\]

\noindent
In the last integrals we have that $-\pi+2t_1-2t_2 \in (2a-\pi, \pi-(\nu+2)a)$ and for $\text{Im} \lambda>0$, $|\lambda|\to \infty,$ the next asymptotic formulas hold:

\begin{equation}\label{29}
\begin{split}
\mathcal{L}_2(\lambda)= O \big(\exp(-i \lambda (\pi-(\nu+2)a))\big), \\
\mathcal{M}_2(\lambda)= O \big(\exp(-i \lambda (\pi-(\nu+2)a))\big).
\end{split}
\end{equation}
For $k > 2$,  the functions $\mathcal{L}_k(\lambda),$ $\mathcal{M}_k(\lambda),$ have less growth than the right-hand side in (\ref {29}). This means that 
\begin{equation}\label{27}
\begin{split}
\mathcal{L}^+ (\lambda) =  O \big(\exp(-i \lambda (\pi-(\nu+2)a))\big), \\
\mathcal{M}^+ (\lambda) =  O \big(\exp(-i \lambda (\pi-(\nu+2)a))\big).
\end{split}
\end{equation}
It follows from (\ref{25},\ref{26},\ref{27}) that for $\text{Im} \lambda>0$, $|\lambda|\to \infty,$ 
\begin{equation*}
\begin{split}
\int_{a}^{\pi-\nu a/2} p(t_1) \exp \big(-i \lambda (-\pi+2t_1-a)\big) dt_1= O \big(\exp(-i \lambda (\pi-(\nu+2)a))\big), \\
\int_{a}^{\pi-\nu a/2} q(t_1) \exp \big(-i \lambda (-\pi+2t_1-a)\big) dt_1=O \big(\exp(-i \lambda (\pi-(\nu+2)a))\big),
\end{split}
\end{equation*}
or, which is the same, as
\begin{equation}\label{28}
\begin{split}
\exp(i \lambda (2\pi-(\nu+1)a))\int_{a}^{\pi-\nu a/2} p(t_1) \exp \big(-2i \lambda t_1)\big) dt_1= O \big(1\big), \\
\exp(i \lambda (2\pi-(\nu+1)a))\int_{a}^{\pi-\nu a/2} q(t_1) \exp \big(-2i \lambda t_1)\big) dt_1= O \big(1\big).
\end{split}
\end{equation}
Let us introduce the functions
\[\begin{split}
F(\lambda)=\exp(i \lambda (2\pi-(\nu+1)a))\int_{\pi-(\nu+1)a/2}^{\pi-\nu a/2} p(t_1) \exp \big(-2i \lambda t_1)\big) dt_1, \\
G(\lambda)=\exp(i \lambda (2\pi-(\nu+1)a))\int_{\pi-(\nu+1)a/2}^{\pi-\nu a/2} q(t_1) \exp \big(-2i \lambda t_1)\big) dt_1.
\end{split}
\]
The functions $F(\lambda),$ $G(\lambda)$ are entire in $\lambda$. One can show that  $F(\lambda)=O \big(1\big),$ $G(\lambda)=O \big(1\big),$  for $\text{Im} \lambda\leq0,$ $|\lambda|\to\infty.$ On the other hand, from (\ref{28}) follows that $F(\lambda)=O \big(1\big),$ $G(\lambda)=O \big(1\big),$  for $\text{Im} \lambda>0,$ $|\lambda|\to\infty$.\\
By Liouville's theorem \cite[page 77]{Conway}, it follows that $F(\lambda)\equiv c_1$ and $G(\lambda)\equiv c_2,$ where $c_1, c_2 $ are constants. 
Since $F(\lambda)=o(1),$ $G(\lambda)=o(1)$ for $\lambda \in \mathbb{R}$ and $|\lambda|\to \infty,$ we conclude that $F(\lambda)\equiv 0$ and $G(\lambda)\equiv0.$ Therefore, it holds 
\[
\int_{\pi-(\nu+1)a/2}^{\pi-\nu a/2} p(t_1) \exp \big(-2i \lambda t_1)\big) dt_1=0, \quad \int_{\pi-(\nu+1)a/2}^{\pi-\nu a/2} q(t_1) \exp \big(-2i \lambda t_1)\big) dt_1=0.
\]
This yields $p(x)=q(x)=0$ a.e. on interval $(\pi-(\nu+1)a/2, \pi-\nu a/2).$

\noindent
\paragraph{Step (3):} Applying Step (2) successively for \(\nu = 0, 1, \ldots, 2N-3\), we obtain that \( p(x) = q(x) = 0 \) a.e. on the interval \( (\pi - (N-1)a, \pi) \). Since \((2a, \pi) \subseteq (\pi - (N-1)a, \pi)\), from Step (1) we conclude that \( p(x) = q(x) = 0 \) a.e. on the interval \( (a, \pi) \). Theorem \ref{thm41} is proved.

\end{proof}

\vspace{0.8cm}
\bibliographystyle{plain}
\bibliography{references}

\begin{thebibliography}{10}

\bibitem{Albeverio}
S.~Albeverio, R.~Hryniv, and Ya.~V. Mykytyuk.
\newblock Inverse spectral problems for {Dirac} operators with summable potentials.
\newblock {\em Russ. J. Math. Phys.}, 12(4):406--423, 2005.

\bibitem{BondarenkoY}
N.~Bondarenko and V.~Yurko.
\newblock An inverse problem for {Sturm}-{Liouville} differential operators with deviating argument.
\newblock {\em Appl. Math. Lett.}, 83:140--144, 2018.

\bibitem{Buterin}
S.~Buterin.
\newblock On the uniform stability of recovering sine-type functions with asymptotically separated zeros.
\newblock {\em Math Notes}, 111:343--355, 2022.

\bibitem{ButerinD}
S.~Buterin and N.~Djuri{\'c}.
\newblock Inverse problems for {Dirac} operators with constant delay: uniqueness, characterization, uniform stability.
\newblock {\em Lobachevskii J. Math.}, 43(6):1492--1501, 2022.

\bibitem{ButerinPY}
S.~A. Buterin, M.~Pikula, and V.~A. Yurko.
\newblock Sturm-{Liouville} differential operators with deviating argument.
\newblock {\em Tamkang J. Math.}, 48(1):61--71, 2017.

\bibitem{Conway}
J.~B. Conway.
\newblock {\em Functions of one complex variable {II}}.
\newblock Graduate texts in mathematics. Springer, New York, NY, 1 edition, 1996.

\bibitem{DjuricB1}
N.~Djuri{\'c} and S.~Buterin.
\newblock On an open question in recovering {Sturm}-{Liouville}-type operators with delay.
\newblock {\em Appl. Math. Lett.}, 113:No 106862, 2021.

\bibitem{DjuricB2}
N.~Djuri{\'c} and S.~Buterin.
\newblock On non-uniqueness of recovering {Sturm}-{Liouville} operators with delay.
\newblock {\em Commun. Nonlinear Sci. Numer. Simul.}, 102:No 105900, 2021.

\bibitem{DjuricB3}
N.~Djuri{\'c} and S.~Buterin.
\newblock Iso-bispectral potentials for {Sturm}-{Liouville}-type operators with small delay.
\newblock {\em Nonlinear Anal., Real World Appl.}, 63:No 103390, 2022.

\bibitem{DjuricVo}
N.~Djuri{\'c} and B.~Vojvodi{\'c}.
\newblock Inverse problem for {Dirac} operators with a constant delay less than half the length of the interval.
\newblock {\em Appl. Anal. Discrete Math.}, 17:249--261, 2023.

\bibitem{Eckhardt}
J.~Eckhardt, A.~Kostenko, and G.~Teschl.
\newblock Inverse uniqueness results for one-dimensional weighted {Dirac} operators.
\newblock In {\em Spectral theory and differential equations: V. A. Marchenko's 90th anniversary collection}, pages 117--133. Providence, RI: American Mathematical Society (AMS), 2014.

\bibitem{Erneux}
T.~Erneux.
\newblock {\em Applied delay differential equations}.
\newblock Springer, Berlin, 2009.

\bibitem{FreilingY}
G.~Freiling and V.~A. Yurko.
\newblock Inverse problems for {Sturm}-{Liouville} differential operators with a constant delay.
\newblock {\em Appl. Math. Lett.}, 25(11):1999--2004, 2012.

\bibitem{Gasymov}
M.~G. Gasymov and T.~T. D\v{z}abiev.
\newblock Solution of the inverse problem by two spectra for the {D}irac equation on a finite interval.
\newblock {\em Akad. Nauk Azerba\u{\i}d\v{z}an. SSR Dokl.}, 22(7):3--6, 1966.

\bibitem{Norkin}
S.~B. Norkin.
\newblock {\em Differential equations of the second order with retarded argument. {Some} problems of the theory of vibrations of systems with retardation. {Translated} from the {Russian} by {L}. {J}. {Grimm} and {K}. {Schmitt}}, volume~31 of {\em Transl. Math. Monogr.}
\newblock American Mathematical Society (AMS), Providence, RI, 1972.

\bibitem{Pikula}
M.~Pikula.
\newblock Determination of a {S}turm-{L}iouville-type differential operator with retarded argument from two spectra.
\newblock {\em Mat. Vesnik}, 43(3-4):159--171, 1991.

\bibitem{PikulaVV}
M.~Pikula, V.~Vladi{\v{c}}i{\'c}, and B.~Vojvodi{\'c}.
\newblock Inverse spectral problems for {Sturm}-{Liouville} operators with a constant delay less than half the length of the interval and {Robin} boundary conditions.
\newblock {\em Results in Mathematics}, 74(1):No 45, 2019.

\bibitem{Polyanin}
A.~D. Polyanin, V.~G. Sorokin, and A.~I. Zhurov.
\newblock {\em Delay ordinary and partial differential equations}.
\newblock Adv. Appl. Math. (Boca Raton). Boca Raton, FL: CRC Press, 2024.

\bibitem{Smith}
H.~L. Smith.
\newblock {\em An introduction to delay differential equations with applications to the life sciences}, volume~57.
\newblock springer New York, 2011.

\bibitem{VladicicBV}
V.~Vladi{\v{c}}i{\'c}, M.~Bo{\v{s}}kovi{\'c}, and B.~Vojvodi{\'c}.
\newblock Inverse problems for {Sturm}-{Liouville}-type differential equation with a constant delay under {Dirichlet}/polynomial boundary bonditions.
\newblock {\em Bulletin of the Iranian Mathematical Society}, 48(4):1829--1843, 2022.

\bibitem{VojvodicDV}
B.~Vojvodi{\'c}, N.~Djuri{\'c}, and V.~Vladi{\v c}i{\'c}.
\newblock On recovering {Dirac} operators with two delays.
\newblock {\em Complex Analysis and Operator Theory}, 18(5):No 105, 2023.

\bibitem{VojvodicKC}
B.~Vojvodi{\'c}, N.~Pavlovi{\'c} Komazec, and F.A. {\c{C}}etinkaya.
\newblock Recovering differential operators with two retarded arguments.
\newblock {\em Bolet{\'\i}n de la Sociedad Matem{\'a}tica Mexicana}, 28(3):No 68, 2022.

\bibitem{VojvodicV}
B.~Vojvodic and V.~Vladicic.
\newblock Recovering differential operators with two constant delays under {Dirichlet}/{Neumann} boundary conditions.
\newblock {\em Journal of Inverse and Ill-posed Problems}, 28(2):237--241, 2020.

\bibitem{VojvodicV2}
B.~Vojvodic and V.~Vladicic.
\newblock On recovering {Sturm}–{Liouville} operators with two delays.
\newblock {\em Journal of Inverse and Ill-posed Problems}, 2024.
\newblock doi.org/10.1515/jiip-2023-0093.

\bibitem{VojvodicVD}
B.~Vojvodi{\'c}, V.~Vladi{\v c}i{\'c}, and N.~Djuri{\'c}.
\newblock Inverse problem for {Dirac} operators with two constant delays.
\newblock {\em Journal of Inverse and Ill-posed Problems}, 32(3):573--586, 2024.

\bibitem{WangY}
F.~Wang and C.-F. Yang.
\newblock Inverse problems for {Dirac} operators with a constant delay less than half of the interval.
\newblock {\em J. Math. Phys.}, 65(3):No 032706, 2024.

\bibitem{WangYBD}
F.~Wang, C.-F. Yang, S.~Buterin, and N.~Djuri{\'c}.
\newblock Inverse spectral problems for {Dirac}-type operators with global delay on a star graph.
\newblock {\em Anal. Math. Phys.}, 14(2):No 24, 2024.

\end{thebibliography}

\medskip
\noindent
{\bf Neboj\v{s}a Djuri\'{c}}\\
Faculty of Electrical Engineering,\\
University of Banja Luka, \\
Bosnia and Herzegovina. \\
E-mail: nebojsa.djuric@etf.unibl.org

\medskip
\noindent
{\bf Biljana Vojvodi\'{c}}\\
Faculty of Mechanical Engineering,\\
University of Banja Luka, \\
Banja Luka, Bosnia and Herzegovina. \\
E-mail: biljana.vojvodic@mf.unibl.org

\end{document}